\documentclass[12pt]{article} 

\usepackage[english]{babel}
\usepackage[latin1]{inputenc}
\usepackage{amsmath}
\usepackage{amsthm}
\usepackage{amssymb}
\usepackage{mathrsfs}
\usepackage{tikz-cd}
\usepackage{hyperref}
\usepackage{comment}
\usepackage{bbm}
\usepackage{caption}

\newtheorem{manualtheoreminner}{Theorem}
\newenvironment{manualtheorem}[1]{%
	\renewcommand\themanualtheoreminner{#1}%
	\manualtheoreminner
}{\endmanualtheoreminner}

\newtheorem{theorem}{Theorem}[section]
\newtheorem{proposition}[theorem]{Proposition}
\newtheorem{corollary}[theorem]{Corollary}
\newtheorem{lemma}[theorem]{Lemma}
\newtheorem{conjecture}[theorem]{Conjecture}

\theoremstyle{definition}

\newtheorem*{introductiondefinition}{Main Definition}
\newtheorem{definition}[theorem]{Definition}
\newtheorem{remark}[theorem]{Remark}
\newtheorem{example}[theorem]{Example}
\newtheorem{question}[theorem]{Question}

\DeclareMathOperator{\kk}{\mathbbm{k}}
\DeclareMathOperator{\curv}{curv}

\DeclareMathOperator{\Tor}{Tor}

\makeatletter
\def\blfootnote{\gdef\@thefnmark{}\@footnotetext}
\makeatother

\title{Graded algebras with cyclotomic Hilbert series}
\date{}
\author{Alessio Borz\`{i}, Alessio D'Al\`{i}}

\begin{document}
	
	\maketitle
	
	\blfootnote{\textup{2020} \textit{Mathematics Subject Classification}. Primary: 13D40; Secondary: 13A02, 16S37, 20M14, 13H10.}
	
	\begin{abstract}
		Let $R$ be a positively graded algebra over a field. We say that $R$ is \emph{Hilbert-cyclotomic} if the numerator of its reduced Hilbert series has all of its roots on the unit circle. Such rings arise naturally in commutative algebra, numerical semigroup theory and Ehrhart theory. If $R$ is standard graded, we prove that, under the additional hypothesis that $R$ is Koszul or has an irreducible $h$-polynomial, Hilbert-cyclotomic algebras coincide with complete intersections. In the Koszul case, this is a consequence of some classical results about the vanishing of deviations of a graded algebra. 
	\end{abstract}
	
	\section{Introduction}
	The \emph{Hilbert series} of a positively graded $\kk$-algebra $R$ is a prominent object in commutative algebra. It encodes the information on how many forms of degree $d$ are contained in $R$ for each possible $d$, and has been the object of intense study since the late nineteenth century. Some simple inductive reasoning shows that the Hilbert series can be expressed as a rational function. Many properties of the graded algebra are reflected into the\footnote{In this paper, when not specified differently, we express the Hilbert series of a graded algebra $R$ as a rational function \emph{reduced to lowest terms}; in particular, it makes sense to speak of \emph{the} numerator.} numerator of such expression. In the last few years, several authors have been investigating the behaviour of the roots of this polynomial \cite{brentiwelker, jochemko, bdgms}, often focusing on the combinatorially interesting case when such roots are all real. For the combinatorial consequences of real-rootedness, we direct the interested reader to the survey \cite{branden}.
	
	The focus of the present article is on those graded algebras whose Hilbert series numerator has all of its roots on the unit circle.
	
	\begin{introductiondefinition}[Definition \ref{def:hilbert-cyclotomic}]
		Let $R$ be a positively graded $\kk$-algebra. We say that $R$ is \emph{Hilbert-cyclotomic} (or simply \emph{cyclotomic}) if the numerator of its reduced Hilbert series is Kronecker, i.e.~has all of its roots on the unit circle.
	\end{introductiondefinition}
	
	Cyclotomic graded algebras have been considered in lattice polytope theory, where they are related to Ehrhart-positivity \cite{liu2019positivity, braun2019h}, and numerical semigroup theory \cite{ciolan2016cyclotomic, moree2014numerical,herrera2018coefficients, borzi2020cyclotomic, sawhney2018symmetric}. Moreover, every graded complete intersection is cyclotomic (see the discussion before Proposition \ref{prop:cyclotomic_ci} for a more precise statement).
	
	Generally speaking, the cyclotomic condition cannot be enough to characterize complete intersections: for instance, Gr\"obner deformation preserves the Hilbert series but not necessarily the complete intersection property. Even under the stronger hypothesis that the given algebra is a graded Cohen--Macaulay domain, there exist cyclotomic algebras which fail to be complete intersections, as shown by Stanley \cite[Example 3.9]{stanley1978hilbert} (see also Example \ref{ex:2dim} herein). However, if we restrict our focus to \emph{numerical semigroup rings}, it is yet unknown whether cyclotomic algebras and complete intersections coincide, as we now explain.
	
	A numerical semigroup $S$ is an additive submonoid of $\mathbb{N}$ with finite complement $\mathbb{N} \setminus S$. For an introduction to numerical semigroups, see \cite{rosales2009numerical}. 
	The \emph{semigroup polynomial} of $S$ is defined as $P_S(x) = 1 + (x-1) \sum_{g \in \mathbb{N}\setminus S} x^g$. It is an easy exercise to check that the semigroup polynomial $P_S(x)$ is the numerator of the reduced Hilbert series of the semigroup ring $\kk[S]$. Ciolan, Garc\'{i}a-S\'{a}nchez and Moree \cite{ciolan2016cyclotomic} call the numerical semigroup $S$ \emph{cyclotomic} if $P_S(x)$ has all of its roots in the unit circle, i.e.~the ring $\kk[S]$ is Hilbert-cyclotomic. The original motivation for this notion comes from the following folklore result in number theory (see for instance \cite[Theorem 1]{moree2014numerical}): if $p$ and $q$ are distinct primes and $\langle p,q \rangle$ is the numerical semigroup generated by $p$ and $q$, then $P_{\langle p,q \rangle}(x) = \Phi_{pq}(x)$, where $\Phi_n(x)$ is the $n$-th cyclotomic polynomial. More generally, if $a$ and $b$ are two coprime integers, then
	$P_{\langle a,b \rangle}(x) = \prod_{n \mid ab,\, n \nmid a, \, n \nmid b} \Phi_n(x).$
	
	
	As we have seen above, if $\kk[S]$ is a complete intersection, then $S$ is cyclotomic. It was verified in \cite{ciolan2016cyclotomic} that every cyclotomic numerical semigroup with Frobenius number up to $70$ is a complete intersection (where the Frobenius number of $S$ is $\max(\mathbb{N} \setminus S)$). This motivates the following conjecture:
	
	\begin{conjecture}[Ciolan, Garc\'{i}a-S\'{a}nchez, Moree \cite{ciolan2016cyclotomic}]\label{conjecture ns}
		A numerical semigroup $S$ is cyclotomic if and only if $\kk[S]$ is a complete intersection.
	\end{conjecture}
	
	
	Now, let us go back to the more general setting where $R$ is a positively graded $\kk$-algebra. In the spirit of Conjecture \ref{conjecture ns} and of a question by Stanley \cite[p.~64]{stanley1978hilbert}, it is of interest to find additional hypotheses under which the cyclotomic condition for $R$ becomes equivalent to being a complete intersection. The main result of this paper shows that this is the case for \emph{Koszul algebras}, a class of quadratic standard graded algebras enjoying many desirable homological properties (for an overview, we refer the interested reader to, e.g., \cite{concadenegrirossi} and \cite{froberg1999koszul}).
	
	\begin{manualtheorem}{A}[Theorem \ref{thm:cyclotomic_koszul}]
		If $R$ is a Koszul algebra, then $R$ is Hilbert-cyclotomic if and only if it is a complete intersection.
	\end{manualtheorem}
	
	Moreover, we prove that complete intersections and cyclotomic algebras coincide also under the assumption that $R$ is standard graded and its $h$-polynomial is irreducible over $\mathbb{Q}$. This is in line with a result in the forthcoming article \cite{borzi2020cyclotomic}: see Question \ref{question:irreducible_conjecture} and the discussion preceding it.
	Recall that the Kronecker polynomials which are irreducible over $\mathbb{Q}$ are precisely the cyclotomic polynomials $\Phi_m(x)$.
	
	\begin{manualtheorem}{B}[Theorem \ref{thm:cyclotomic_standard_graded}]
		Let $R$ be a standard graded algebra. Then $h(R, x) = \Phi_m(x)$ if and only if $m$ is prime and $R$ is a hypersurface of degree $m$.
	\end{manualtheorem}

	\section{Hilbert-cyclotomic algebras}\label{sec:cyclotomic}
	
    In this paper, a \emph{graded algebra} will always be a commutative finitely generated $\mathbb{N}$-graded algebra $R = \bigoplus_{i \in \mathbb{N}}R_i$ with $R_0 = \kk$, where $\kk$ is a field. We will write the reduced Hilbert series of $R$ as $H(R, x) = N_R(x)/D_R(x)$. If $R$ is generated by its degree 1 part, we will say that $R$ is \emph{standard graded}. In this case, as is customary, we will call $N_R$ the \mbox{\emph{$h$-polynomial}} of $R$ and denote it by $h(R,x)$.
	
	
	
	\begin{definition}\label{def:hilbert-cyclotomic}
		Let $R$ be a graded algebra. We say that $R$ is \emph{Hilbert-cyclotomic} (or simply \emph{cyclotomic}) if the numerator of its reduced Hilbert series is Kronecker, i.e.~has all of its roots on the unit circle.
	\end{definition}
	
	\begin{remark}
		Besides the connection to numerical semigroups highlighted in the introduction, the cyclotomic condition has also been studied in Ehrhart theory. Braun and Liu \cite{liu2019positivity,braun2019h} call a polytope \emph{$h^*$-unit-circle-rooted} if its Ehrhart ring is Hilbert-cyclotomic. Remarkably, such polytopes are Ehrhart-positive \cite[Corollary 1.4]{braun2019h}.
	\end{remark}
	
	A degree $s$ polynomial $f(x) = \sum_{i=0}^s a_i x^i$ with integer coefficients is said to be \emph{palindromic} if $a_i = a_{s-i}$ for every $i \in \{ 0,\dots,s \}$, or equivalently if $f(x) = x^s f(1/x)$. For every $n>1$, one has that the cyclotomic polynomial $\Phi_n(x) := \prod_{(j,n) = 1} \left( x - e^{2 \pi i j / n} \right)$ is palindromic. Since a Kronecker polynomial is a product of cyclotomic polynomials and the palindromic property is preserved under taking products, it follows that every Kronecker polynomial $f$ with $f(1) \neq 0$ is palindromic.
	
	Now let $R$ be a graded algebra of Krull dimension $d$. Since the order of the pole of $H(R,x)$ at $x=1$ equals $d$ (see for instance \cite[Chapter 11]{atiyah2016commutative}), it follows that $N_R(1) \neq 0$. Thus, from the above observations we infer that
	
	\begin{remark} \label{rem:cyclotomic_palindromic}
		If the graded algebra $R$ is cyclotomic, then the numerator $N_R$ of its reduced Hilbert series is palindromic.
	\end{remark}
	
	The rest of this section is devoted to a quick exploration of how the cyclotomic condition relates to complete intersections and Gorenstein algebras, as summarized by the following diagram:
	
	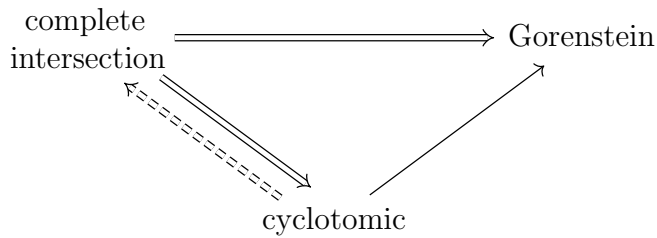
\begin{figure}[htp]
	\captionsetup{width=0.8\textwidth}
		\[
		\begin{tikzcd}
			\text{\parbox{2cm}{\centering complete \\ intersection}}
			\arrow[Rightarrow]{rr}
			\arrow[Rightarrow, shorten <=5pt, shorten >=-2pt, shift left=5pt]{dr}
			&&
			\text{Gorenstein}
			\\[0.8cm]
			& \text{cyclotomic}
			\arrow[Rightarrow, dashed, shorten <=5pt, shorten >=-5pt, shift left=5pt]{ul}
			\arrow[shorten >=5pt]{ur}
			&
		\end{tikzcd}
		\]
		\caption{The single arrow holds for Cohen-Macaulay domains (Proposition \ref{prop:cyclotomic implies Gorenstein}), whereas the dashed arrow is the subject of Conjectures \ref{conjecture ns} and \ref{conj:ns_algebraic} and holds for standard graded algebras that are Koszul (Theorem \ref{thm:cyclotomic_koszul}) or have irreducible $h$-polynomial (Theorem \ref{thm:cyclotomic_standard_graded}).}

	\end{figure}
	
	
	
	Let $R$ be a (graded) complete intersection, i.e.~a quotient of a positively graded polynomial ring $S = \kk[x_1,\dots,x_n]$ by an ideal generated by a homogeneous regular sequence $f_1, \dots, f_e$. Setting $d_i := \deg x_i$ and $m_j := \deg f_j$, one shows that the Hilbert series of $R$ can be written in the (non-reduced) form
	\[ H(R,x) = \frac{(1-x^{m_1}) \dots (1-x^{m_e})}{(1-x^{d_1}) \dots (1-x^{d_n})},\]
	see for instance \cite[Corollary 3.3]{stanley1978hilbert}. Hence, it follows that
	
	\begin{proposition} \label{prop:cyclotomic_ci}
		Every graded complete intersection is cyclotomic.
	\end{proposition}
	
	The converse of Proposition \ref{prop:cyclotomic_ci} does not hold even under the hypothesis that the given algebra is a Cohen--Macaulay standard graded domain, as already observed by Stanley \cite[Example 3.9]{stanley1978hilbert}. Examples of cyclotomic non-complete intersection Cohen--Macaulay standard graded domains can be found in any dimension $d \geq 2$: it is enough to adjoin variables to the following example provided by Aldo Conca.
	
	\begin{example} \label{ex:2dim}
		Let $R = \kk[s^8, s^6t^2, s^5t^3, s^3t^5, t^8] \subseteq \kk[s, t]$. Then $R$ is a \mbox{2-dimensional} standard graded domain which is Cohen--Macaulay in every characteristic (one checks via some characteristic-free Gr\"obner basis computation that the system of parameters $\{s^8, t^8\}$ is a regular sequence for $R$). Since the $h$-polynomial of $R$ is $(1+x)^3$, one has that $R$ is cyclotomic; however, it is not a complete intersection.
	\end{example}

	It turns out that investigating the case of one-dimensional domains (not necessarily standard graded) is essentially equivalent to solving Conjecture \ref{conjecture ns}. In fact, from \cite[Proposition 3.1]{stanley1991hilbert} every graded domain of Krull dimension one over an algebraically closed field $\kk$ is isomorphic to the semigroup algebra $\kk[\Gamma]$ of some additive submonoid $\Gamma$ of $\mathbb{N}$. Further, every submonoid $\Gamma$ of $\mathbb{N}$ is isomorphic to the numerical semigroup $S = \Gamma/\gcd(\Gamma)$ \cite[Proposition 2.2]{rosales2009numerical}. Hence, we can reformulate Conjecture \ref{conjecture ns} in purely algebraic terms:
	
	\begin{conjecture}[Conjecture \ref{conjecture ns}, algebraic version] \label{conj:ns_algebraic}
		Every cyclotomic graded domain of Krull dimension one over an algebraically closed field is a complete intersection.
	\end{conjecture}

	We close this section by discussing the relation between the cyclotomic condition and the Gorenstein property. A famous theorem by Stanley \cite[Theorem 4.4]{stanley1978hilbert} states that a Cohen--Macaulay graded domain $R$ is Gorenstein if and only if $N_R(x)$ is palindromic. Recalling Remark \ref{rem:cyclotomic_palindromic}, we hence obtain the following result:
	
	\begin{proposition}\label{prop:cyclotomic implies Gorenstein}
		Let $R$ be a Cohen--Macaulay graded domain. If $R$ is cyclotomic, then $R$ is Gorenstein.
	\end{proposition}
	
	Proposition \ref{prop:cyclotomic implies Gorenstein} generalizes both \cite[Theorem 5]{moree2014numerical} for numerical semigroups 
	and \cite[Corollary 2.2.9]{liu2019positivity} for lattice polytopes. The converse of Proposition \ref{prop:cyclotomic implies Gorenstein}, however, does not hold even in these more specific settings. One can consider for instance the semigroup rings of the numerical semigroups \mbox{$S_k = \langle k,k+1,\dots,2k-2 \rangle$} with $k \geq 5$ \cite{herrera2018coefficients,sawhney2018symmetric} or Ehrhart rings of Birkhoff polytopes \cite[Section 5.1.2]{liu2019positivity}.

	\section{Cyclotomic Koszul algebras} \label{sec:koszul}
	The goal of this section is to prove that, for Koszul algebras, the cyclotomic property characterizes complete intersections. 
	We begin by some definitions.
	
	\begin{definition}
		Let $R$ be a standard graded $\kk$-algebra. We say that $R$ is \emph{Koszul} if the minimal graded free resolution of $\kk$ as an $R$-module is linear, i.e.~$\Tor^R_i(\kk, \kk)_j = 0$ whenever $i \neq j$.
	\end{definition}
	
	\begin{definition}
		Let $R$ be a graded $\kk$-algebra. The \emph{Poincar\'e series} of $\kk$ as an $R$-module is $\mathcal{P}(R, x) = \sum_{i=0}^{+\infty}\beta^R_i(\kk)x^i$, where $\beta_i^R(\kk) := \dim_{\kk}\Tor^R_i(\kk, \kk)$.
	\end{definition}
	
	
	
	
	
	
	The following remark can be found for instance in \cite[Remark 1]{moree2000automata}.
	
	\begin{remark} \label{rmk:cyclotomic_sequence}
		Given a formal series $P(x) = 1 + \sum_{i=1}^{+\infty}a_ix^i$ with $a_i \in \mathbb{Z}$, there exist unique integers $e_i \in \mathbb{Z}$ such that
		\begin{equation}\label{cyclotomic factorization}
		P(x) = \prod_{i=1}^{+\infty}(1-x^i)^{e_i}.
		\end{equation}
		If $P(x)$ is just a polynomial, then it is Kronecker if and only if $e_i = 0$ for $i \gg 0$. For a proof of this fact, see \cite[Lemma 12]{ciolan2016cyclotomic}.
	\end{remark}
	
	The factorization in \eqref{cyclotomic factorization} was used by Ciolan, Garc{\'i}a-S{\'a}nchez and Moree \cite{ciolan2016cyclotomic} to define the \emph{cyclotomic exponent sequence} of a numerical semigroup. This notion can be generalized as follows.
	
	\begin{definition}
		Let $R$ be a graded algebra and let $N_R(x)$ be the numerator of its reduced Hilbert series. Since $N_R(0) = 1$, we can factor $N_R$ as in \eqref{cyclotomic factorization}. The integers $e_i$ will be called the cyclotomic exponent sequence of $R$ and will be denoted by $e_i(R)$.
	\end{definition}

	A consequence of Remark \ref{rmk:cyclotomic_sequence} is the following equivalence.
	
	\begin{corollary}\label{cyclotomic iff cyc exp seq 0}
		Let $R$ be a graded algebra. The following conditions are equivalent:
		\begin{enumerate}
			\item $R$ is cyclotomic;
			\item $e_i(R) = 0$ for $i \gg 0$.
		\end{enumerate}
	\end{corollary}
	
	\begin{definition}
		Let $R$ be a graded $\kk$-algebra and let $\mathcal{P}(R, x)$ be the Poincar\'e series of the residue field $\kk$ as an $R$-module. We write 
		\[ \mathcal{P}(R,x) = \frac{\displaystyle \prod_{i=1}^{+\infty} (1+x^{2i-1})^{\varepsilon_{2i-1}}}{\displaystyle \prod_{i=1}^{+\infty} (1-x^{2i})^{\varepsilon_{2i}}}\]
		and call the integers $(\varepsilon_i)_{i \in \mathbb{N}}$ so obtained the sequence of \emph{deviations} of $R$.
	\end{definition}
	
	An interesting feature of deviations is their ability to tell whether or not $R$ is a complete intersection. The strongest version of this result is Halperin's rigidity theorem, see for instance \cite[Theorem 7.3.4]{avramov1998infinite}. For our aims, however, a weaker statement originally due to Gulliksen \cite{gulliksen1971ci} will suffice (see also \cite[Theorem 7.3.3]{avramov1998infinite}):
	
	\begin{theorem} \label{thm:gulliksen_eventually_zero}
		Let $R$ be a graded algebra. The following conditions are equivalent:
		\begin{enumerate}
			\item $R$ is a complete intersection;
			\item $\varepsilon_i(R) = 0$ for $i \gg 0$.
		\end{enumerate}
	\end{theorem}
	
	Corollary \ref{cyclotomic iff cyc exp seq 0} and Theorem \ref{thm:gulliksen_eventually_zero} exhibit a formal similarity. Such a similarity becomes substantial when $R$ is a Koszul algebra, as the following result shows.
	
	\begin{proposition} \label{prop:koszul_cyclotomic_deviations}
		Let $R$ be a Koszul algebra of Krull dimension $d$. Then
		\[ e_i(R) =
		\begin{cases}
		-\varepsilon_1(R) + d & i = 1 \\
		(-1)^i \varepsilon_i(R) & i > 1.
		\end{cases}
		\]
		In particular, $e_i(R) = 0$ for $i \gg 0 \Longleftrightarrow \varepsilon_i(R) = 0$ for $i \gg 0$.
	\end{proposition}
	\begin{proof}
		Note that $N_R(x) = h(R,x) = (1-x)^d H(R,x)$ because $R$ is standard graded. Further, since $R$ is Koszul, from \cite[Theorem 1]{froberg1999koszul} we have that $H(R,x) \mathcal{P}(R,-x) = 1$. Now write
		\[
		\begin{split}
		& \phantom{=} \prod_{i=1}^{+\infty}(1-x^i)^{e_i(R)} = h(R, x) = (1-x)^d H(R,x) = \frac{(1-x)^d}{\mathcal{P}(R, -x)} = \\
		& = (1-x)^d \frac{\displaystyle\prod_{j=1}^{+\infty}(1-x^{2j})^{\varepsilon_{2j}(R)}}{\displaystyle\prod_{j=1}^{+\infty}(1-x^{2j-1})^{\varepsilon_{2j-1}(R)}} = (1-x)^d \prod_{i=1}^{+\infty}(1-x^i)^{(-1)^i\varepsilon_i(R)}
		\end{split}
		\]
		and the claim follows.
	\end{proof}

	\begin{theorem} \label{thm:cyclotomic_koszul}
		If $R$ is a Koszul algebra, then $R$ is cyclotomic if and only if it is a complete intersection.
	\end{theorem}
	\begin{proof}
		The result follows directly from Proposition \ref{prop:koszul_cyclotomic_deviations}, Corollary \ref{cyclotomic iff cyc exp seq 0} and Theorem \ref{thm:gulliksen_eventually_zero}.
	\end{proof}
	
	\begin{remark}
		One may also show that, if $R$ is Koszul and cyclotomic, then the Betti numbers of the residue field $\kk$ as an $R$-module ``do not grow too fast'', i.e.~it holds that
		\[\curv_R(\kk) := \limsup_{n \to +\infty}{\sqrt[n]{\beta_n^R(\kk)}} \leq 1.\]
		By \cite[Corollary 8.2.2]{avramov1998infinite}, this implies that $R$ is a complete intersection; however, the proof of \cite[Corollary 8.2.2]{avramov1998infinite} still relies on deviations.
	\end{remark}
	
	\section{Algebras with irreducible $h$-polynomial}\label{sec:irreducible}
	
	Let $R$ be a cyclotomic graded algebra and assume that $N_R$ is irreducible over $\mathbb{Q}$. This means that $N_R(x) = \Phi_m(x)$ 
	for some $m \in \mathbb{N}$. Under this condition, in the case when $R = \kk[S]$ for some numerical semigroup $S$ (and hence $N_R$ equals the semigroup polynomial $P_S$), it is proved in \cite{borzi2020cyclotomic} that then $S = \langle p,q \rangle$ for some primes $p \neq q$, and consequently $m = pq$. Since each numerical semigroup of the form $\langle p,q \rangle$ is a complete intersection, this implies in particular that Conjecture \ref{conjecture ns} holds true when $P_S$ is irreducible. This prompts the following questions:
	
	\begin{question}\label{question:irreducible_conjecture}
		Let $R$ be a graded algebra and assume that the numerator $N_R$ of its reduced Hilbert series is irreducible. Is it true that $R$ is cyclotomic if and only if it is a complete intersection?
	\end{question}
	
	\begin{question} \label{question:irreducible_cyclotomic}
		Which cyclotomic graded algebras $R$ have a Hilbert series whose numerator is irreducible, i.e.~$N_R(x) = \Phi_m(x)$ for some $m$?
	\end{question}
	
	The aim of this section is to answer both of the above questions in the case when $R$ is standard graded.
	
	\begin{theorem} \label{thm:cyclotomic_standard_graded}
		Let $R$ be a standard graded algebra. Then $h(R, x) = \Phi_m(x)$ if and only if $m$ is prime and $R$ is a hypersurface of degree $m$.
	\end{theorem}
	
	As a consequence, Question \ref{question:irreducible_conjecture} has a positive answer when $R$ is standard graded. 
	To prove Theorem \ref{thm:cyclotomic_standard_graded}, we will need some auxiliary results. First, we recall some basic properties of cyclotomic polynomials.
	
	\begin{lemma} \label{lem:cyclotomic_at_1} \ 
		\begin{itemize}
			\item[\emph{(a)}] Let $m > 1$. Then $\Phi_m(1) =
			\begin{cases}
			p & \textrm{if } m = p^k \textrm{ for some prime }p\\
			1 & \textrm{otherwise.}
			\end{cases}
			$
			\item[\emph{(b)}] For any prime $p$ and any $k \geq 2$, one has that $\Phi_{p^k}(x) = \Phi_{p}(x^{p^{k-1}})$.
		\end{itemize}
	\end{lemma}
	
	The following lemma is a generalization of \cite[Theorem 3.6]{stanley1978hilbert}.
	
	\begin{lemma} \label{lem:hypersurface}
		Let $R$ be a standard graded $\kk$-algebra such that 
		\begin{equation} \label{eq:reduced_hilbert}
		H(R, x) = \frac{1+x+x^2+\ldots+x^{s-1}}{(1-x)^d}
		\end{equation}
		for some $s>1$, $d \geq 0$. Then $R = \kk[x_1, \ldots, x_{d+1}]/(f)$ for some homogeneous polynomial $f$ of degree $s$.
	\end{lemma}
	
	\begin{proof}
		Write $R$ as the quotient of a standard graded polynomial ring $S = \kk[x_1, \ldots, x_n]$ by a homogeneous ideal $I \subseteq (x_1, \ldots, x_n)^2$. Here $n$ is the embedding dimension of $R$, which is strictly greater than $d$ since $R$ is not a polynomial ring itself. It follows from Hilbert's syzygy theorem that the Hilbert series of $R$ can be written in a non-reduced way as
		\begin{equation} \label{eq:K_hilbert}
		H(R, x) = \frac{\mathcal{K}(R, x)}{(1-x)^n},
		\end{equation}
		where $\mathcal{K}(R, x) = \sum_{i,j} (-1)^i \beta_{i,j}^S(R) x^j$ and $\beta_{i,j}^S(R) := \dim_{\kk} \Tor^S_i(R, \kk)_j$ is the $(i,j)$-th Betti number of $R$ as an $S$-module. Comparing Equations \eqref{eq:reduced_hilbert} and \eqref{eq:K_hilbert}, we have that \[\mathcal{K}(R, x) = (1+x+x^2+\ldots+x^{s-1})(1-x)^{n-d} = (1-x^s)(1-x)^{n-d-1}.\]
		Now, if $n-d-1 > 0$, it follows that the coefficient of $x$ in $\mathcal{K}(R, x)$ is nonzero, which is impossible since $I$ does not contain any linear form. Hence, $n=d+1$ and $\mathcal{K}(R, x) = 1-x^s$. It is left as an exercise to the reader to check that $I$ must then be minimally generated by a single homogeneous polynomial of degree $s$. 
	\end{proof}
	
	\begin{lemma}\label{lem:numerator not 1}
		Let $R$ be a standard graded $\kk$-algebra of Krull dimension $d$. If the $h$-polynomial of $R$ is palindromic of even degree $s$, then $h(R,1) > 1$.
	\end{lemma}
	\begin{proof}
		The integer $h(R, 1)$ is the multiplicity of $R$ \cite[Definition 4.1.5, Corollary 4.1.9]{bruns1998cohen}, and as such it is positive. Suppose by contradiction that $h(R,1)$ equals $1$, and write $R$ as $\kk[x_1,\dots,x_n]/I$ for some homogeneous ideal $I$. Since the Hilbert series stays the same when passing to the initial ideal, we can assume without loss of generality that $I$ is a monomial ideal. Moreover, since polarization preserves the $h$-polynomial, we can further assume that $I$ is a squarefree monomial ideal; hence, $I$ is the Stanley--Reisner ideal of some simplicial complex $\Delta$.
		
		Let $D-1$ be the dimension of $\Delta$, and set $h(R,x) = \sum_{i=0}^{s}h_ix^i$.
		By construction, the $h$-vector of $\Delta$ is $(h_0, h_1, \ldots, h_D)$, where $h_i = 0$ if $i > s$. Knowing the $h$-vector of $\Delta$ gives us access to its $f$-vector $(f_{-1}, f_0, \ldots, f_{D-1})$, where $f_i$ is the number of $i$-dimensional faces of $\Delta$. As shown for instance in \cite[Corollary 1.15]{miller2004combinatorial}, the transformation is given by
		\begin{equation} \label{eq:h_f_Delta}
		\sum_{i=0}^{D}f_{i-1}(x-1)^{D-i} = \sum_{i=0}^{D}h_ix^{D-i}.
		\end{equation}
		In particular, $f_{D-1} = \sum_{i=0}^{D}h_i = \sum_{i=0}^{s}h_i = h(R,1) = 1$. Since $\Delta$ contains a $(D-1)$-dimensional face, there should be at least $D$ $(D-2)$-dimensional faces, that is $f_{D-2} \geq D$. Substituting $x-1$ by $y$ inside Equation \eqref{eq:h_f_Delta}, we find that $f_{D-2}$ is the coefficient of $y$ in $\sum_{i=0}^{D}h_i(y+1)^{D-i}$. Hence,
		\[\begin{split}f_{D-2} &= Dh_0 + (D-1)h_1 + \ldots + (D-s)h_s\\
		&= (D-s)\sum_{i=0}^{s}h_i + sh_0 + (s-1)h_1 + \ldots + h_{s-1}\\
		&= (D-s) + \frac{s}{2}\sum_{i=0}^{s}h_i = D - \frac{s}{2} < D,
		\end{split}\]
		where the third equality comes from the fact that $h(R,x)$ is palindromic of even degree.
	\end{proof}
	
	\begin{proof}[Proof of Theorem \ref{thm:cyclotomic_standard_graded}]
		The ``if'' part is clear. Let us prove the ``only if''. Suppose that $R = \kk[x_1, \ldots, x_n]/I$ for some homogeneous ideal $I$ contained in $(x_1, \ldots, x_n)^2$. By hypothesis we have that
		\begin{equation} \label{eq:hilbert_hypothesis2}
		H(R, x) = \frac{\Phi_m(x)}{(1-x)^d}
		\end{equation}
		for some $m > 1$ and $d = \dim R \geq 0$. Since $\Phi_m(x)$ is palindromic of even degree, applying Lemma \ref{lem:numerator not 1} yields that $\Phi_m(1) \neq 1$. It follows from part (a) of Lemma \ref{lem:cyclotomic_at_1} that $m=p^k$ for some prime $p$ and $k \geq 1$. Now assume that $k > 1$ and let $q = p^{k-1}$. By part (b) of Lemma \ref{lem:cyclotomic_at_1}, one has that $\Phi_m(x) = \Phi_p(x^q) = 1 + x^q + x^{2q} + \ldots + x^{(p-1)q}$.
		Expanding Equation \eqref{eq:hilbert_hypothesis2} at $x=0$, we get that the standard graded algebra $R$ contains $d$ forms of degree $1$ and $\binom{d+q-1}{q} + 1$ forms of degree $q$, which is impossible. Hence, $m=p$. Now apply Lemma \ref{lem:hypersurface}.
	\end{proof}

	\section*{Acknowledgements}
	The second-named author was supported by the EPSRC grant EP/R02300X/1. The authors are grateful to Aldo Conca and Diane Maclagan for some useful advice on the structure of the present paper; Aldo Conca has also provided Example \ref{ex:2dim}. The authors wish to thank also Benjamin Braun and Winfried Bruns for some fruitful discussions and insights.
	Many computations and examples were carried out via the computer algebra systems Macaulay2 \cite{M2} and Normaliz \cite{Normaliz}.
	
	\bibliographystyle{abbrv}
	\bibliography{CyclotomicReferences}
	
	\noindent
	{\scshape Alessio Borz\`{i}} \quad \texttt{Alessio.Borzi@warwick.ac.uk}\\
	{\scshape  Mathematics Institute, University of Warwick, Coventry CV4 7AL, United Kingdom.}
	
	\medskip
	\noindent
	{\scshape Alessio D'Al\`{i}} \quad \texttt{Alessio.D-Ali@warwick.ac.uk}\\
	{\scshape  Mathematics Institute, University of Warwick, Coventry CV4 7AL, United Kingdom.}
	
\end{document}